\documentclass{dcds}
\usepackage{psfrag,epsf,epsfig,color,comment}
\def\currentvolume{12}
 \def\currentissue{3}
  \def\currentyear{2005}
   \def\currentmonth{March}
    \def\ppages{465--480}

\numberwithin{equation}{section}
 \newtheorem{theorem}{Theorem}[section]
\newtheorem{definition}[theorem]{Definition}

\newtheorem{lemma}[theorem]{Lemma}
\newtheorem{remark}[theorem]{Remark}
\newtheorem{proposition}[theorem]{Proposition}
\newtheorem{claim}[theorem]{Claim}
\newcommand{\dl}{\frac{\partial}{\partial x_l}}
\newcommand{\di}{\frac{\partial}{\partial x_i}}
\newcommand{\dk}{\frac{\partial}{\partial y_k}}

\thanks{I am thankful to C. Pugh and M. Hirsch for helpful and instructive
discussions}

\begin{document}

\title{Homoclinic tangencies in ${\mathbb{R}}^n$}

\author
{Victoria Rayskin}

\address{Victoria Rayskin \hfill\break
       Department of Mathematics\\
       520 Portola Plaza\\
       Box 951555\\
       University of California\\
       Los Angeles, CA 90095-1555, USA}
\email{vrayskin\symbol{64}math.ucla.edu}

\date{}


\begin{abstract}

Let $f: M \rightarrow M$ denote a diffeomorphism of a smooth manifold $M$. Let $p \in M$ be its hyperbolic fixed point with
stable and unstable manifolds $W_S$ and $W_U$ respectively. Assume that $W_S$
is a curve. Suppose that $W_U$ and
$W_S$ have a {\it degenerate homoclinic crossing} at a point $B\neq p$, i.e.,
they cross at $B$ tangentially with a finite order of contact.

It is shown that, subject to $C^1$-linearizability and certain conditions on
the invariant manifolds, a transverse homoclinic crossing will arise
arbitrarily close to $B$. This proves the existence of a horseshoe structure
arbitrarily close to $B$, and extends a similar planar
result of Homburg and Weiss \cite{HW}.

\end{abstract}

\maketitle

\section{Introduction}

The celebrated Theorem of
Birkhoff-Smale (see, for instance, \cite{S}) provides a method for rigorously concluding the existence of a
horseshoe structure. The basic assumption of this theorem is the presence of a
transverse homoclinic point.

\begin{theorem}[Birkhoff-Smale]\label{BS}
Let $f: {\mathbb{R}}^n \to {\mathbb{R}}^n $ be a diffeomorphism with a hyperbolic fixed point $p$. If $q \not= p$ is a
transverse homoclinic point, then there exist
 a hyperbolic invariant set $\Lambda$, a set of bi-infinite sequences ${\Sigma}$
and a shift map $\sigma$ such that for some $m$ the following diagram commutes:
\[
\begin{array}{ccc}
\Lambda & \stackrel{f^m}{\rightarrow}& \Lambda \\
h\downarrow \ \  & \nonumber & \ \ \downarrow h \\
\Sigma & \stackrel{\sigma }{\rightarrow} & \Sigma
\end{array}
\]
Here $\Sigma = \{... A_{-2} A_{-1} . A_0 A_1 A_2 ... \} $, $A_i \in \{0,...,N \}$, \newline
$ \sigma \{... A_{ -2} A_{ -1} . A_{ 0} A_{ 1} A_{ 2} ... \} =
\{ ... A_{ -2} A_{ -1} A_{ 0} . A_{ 1} A_{ 2} ... \} $
\nopagebreak
 and $h$ is a homeomorphism mapping $\Lambda $ onto $\Sigma $.
\end{theorem}

The assumption of transversality is not easy to verify for a concrete dynamical
system. There were many attempts to remove this assumption.
(See  K. Burns and H. Weiss~\cite{BW}, A. J. Homburg and H. Weiss~\cite{HW}, M. Hirsch~\cite{Hi},
R. Churchill and D. Rod~\cite{CR}, Gavrilov and Shilnikov~\cite{GS1},~\cite{GS2}, Gonchenko and
Shilnikov~\cite{GS}). These papers
address the question of existence of horseshoes arbitrarily close to a
homoclinic point in {\it planar} dynamics.

The object of the papers of Hirsch~\cite{Hi} and Churchill and Rod~\cite{CR}
was to show that, in the planar case, the manifolds $W_S$ and $W_U$ cross
transversely at some other point, arbitrarily close to the original crossing,
if the original crossing is two-sided (Figure~\ref{types_of_crossing}, case IV)
and of finite contact.

Churchill and Rod~\cite{CR} imposed the additional assumptions that $f$ is {\it
analytic} and {\it area-preserving}. In that case, by Moser's Theorem \cite{M},
this diffeomorphism  can be assumed to be in the Birkhoff normal form, and any
branch of the stable (unstable) manifold can be re-parameterized so that $f^n
(W_U) \cap W_S$ contains a positive orbit $q', fq', f^{2}q',..., f^{k}q'$ that
lies on a hyperbola. Then, direct calculations imply that $q'$ is a transverse
homoclinic point. But area preservation is a strong restriction. Furthermore,
analyticity allows one to
consider the case of a finite order
of contact only.

Hirsch~\cite{Hi},
replaced the above assumptions with the smooth linearizability assumption (near
$p$). If we substitute $L$, the linear part of $f$, for $f$, the order of
contact will not be changed, and we can calculate the ``slope'' between two
manifolds at the point $q'$, where $L^n (W_U)$ crosses $W_S$. If the crossing
at $q'$ is transverse, then $f^n (W_U) = \Phi {L}^n {\Phi}^{-1} (W_U)$ also
crosses $W_S$ non-degenerately, where $\Phi $ is the $C^1$ linearizing map.

Homburg and Weiss~\cite{HW} considered a
surface diffeomorphism and its homoclinic {\it one-sided}
(Figures~\ref{types_of_crossing}, \ref{newhousePic}) tangency of arbitrary high (possibly infinite)
order of contact. They proved that
some power of this diffeomorphism has the full shift on two symbols as a
topological factor. Then, Katok's theorem \cite{K1}, \cite{K2} implies that the
map possesses a horseshoe arbitrarily close to homoclinic tangency.

Important cases of homoclinic tangencies were also described in works of
Shilnikov, Gonchenko and Gavrilov (see \cite{GS},
\cite{GS1}, \cite{GS2}). They considered {\it one-sided}
planar homoclinic tangencies.

Gavrilov and Shilnikov \cite{GS1}, \cite{GS2} established a topological
conjugacy (on a closed invariant set) between a surface diffeomorphism having a
dissipative hyperbolic periodic point with certain types of quadratic ({\it
one-sided}) tangencies (Figures~\ref{types_of_crossing}, \ref{newhousePic}) and
the full shift on two symbols.

Gonchenko and Shilnikov considered higher orders of contact. Homburg and Weiss write: ``Gonchenko claims that
a proof of this result for finite order tangencies appeared in his unpublished
(Russian) thesis in 1984. The result for finite order tangencies was announced
by Gonchenko and Shilnikov~\cite{GS} in 1986, but we are unable to find any
proof in the literature.'' We were unable to find such a proof in the
existing literature either.
\bigskip

Homburg and Weiss, in their paper~\cite{HW}, state an Open Problem, asking whether their results could be extended to dimensions
higher than two.
In the present paper, we consider a diffeomorphism $f$ of a smooth
manifold $M$ (of an arbitrary dimension) with a hyperbolic fixed point. We assume that one of the
invariant manifolds of $f$ is a curve, that this curve has finite
tangential contact with the
other invariant manifold, and that the contact is not of Newhouse type (see Definition~\ref{newhouseDef} and
Figure~\ref{newhousePic}). We prove (Theorem~\ref{main-thm}) that, subject to
$C^1$-linearizability of $f$ and the existence of a graph portion of the
invariant manifold, a transverse homoclinic crossing will arise arbitrarily
close to the tangential crossing. This implies the existence of a horseshoe
arbitrarily close to the tangential crossing.

To prove this result, we consider (Section 3) tangent subspaces of stable and unstable
manifolds at a certain sequence of their intersection points. We show that, at
some intersection points, a pair of the tangent subspaces spans the entire
${\mathbb{R}}^n$. In Section 2 we present basic definitions and lemmas that are used
in Section 3. In particular, we prove tangential $\lambda$-Lemma~\ref{lemma}, essential for the proof of the main result.
Also, in Section 2, we present a new definition of the order of contact of two manifolds. This definition does not require high
differentiability and plays a key role in our approach.

\section{Definitions and Preliminaries}\label{2}

In this section we give basic definitions and state several preliminary lemmas
to be used in the proof of the main result. The references for this Section are
V. I. Arnold, S. M. Guse\u\i n-Zade~\cite{AVG-Z} and A. N. Varchenko, K. Burns and H. Weiss~\cite{BW}, M. Hirsch~\cite{Hi},
A. J. Homburg and H. Weiss~\cite{HW}, S. Newhouse and J. Palis~\cite{NP}, J. Palis~\cite{Pa}, V. Rayskin~\cite{R},
 R. Uribe-Vargas~\cite{U-V}.

Throughout this section, we will be considering two submanifolds in a certain
ambient smooth manifold (of arbitrary dimension). Suppose two such submanifolds
meet at an isolated point $A$. We will discuss the contact properties of these
submanifolds at point $A$. Since these properties are local, we may assume without loss of
generality that the ambient manifold is just ${\mathbb R}^n$.

First, assume that each manifold is a curve. The order of contact of two smooth
curves has been defined and studied in works of Arnold, Guse\u\i
n-Zade~\cite{AVG-Z} and A. N. Varchenko, Burns and Weiss~\cite{BW}, Hirsch~\cite{Hi}, Homburg and Weiss~\cite{HW}, Uribe-Vargas~\cite{U-V}. Their definitions of the
$l$-th order of contact require high differentiability; i.e. $l$ is the maximal
integer, such that the first $l-1$ derivatives of the two curves coincide at
the point of their contact.  Our definition below does not require high  differentiability and thus, allows us to apply $C^1$-linearization in our main
Theorem~\ref{main-thm}.

\begin{definition}\label{def-ooc-curve} \rm
Let ${\gamma}_i$ ($i=1,2$) denote two immersed $C^1$-curves in ${\mathbb{R}}^n$
. Suppose the two curves meet at an isolated point $A$. Then we will say that a
number $l$ is the {\it order of contact} of the curve $\gamma_1$ with the curve $\gamma_2$ at the point $A$ if there exist positive real numbers $m$
and $M$ such that for all points $x$ on $\gamma_1$ sufficiently close to $A$
\[
m\leq
 \frac
{d(x,\gamma_2)}{\|x - A  \|^l }
 \leq M \;.
\]
\end{definition}

\begin{proposition}\label{symmetry}
If the curves are $C^1$-smooth, then the definition of the order of contact is symmetric with respect
to $\gamma_1$ and $\gamma_2$.
\end{proposition}

\begin{proof}{}
Since the curves are $C^1$-smooth, we can assume that $\gamma_2$ is the $X$-axis (in some coordinate system), and $\gamma_1$ is a graph of a $C^1$ function $f$ in this system. Let the order of contact of $\gamma_1$ with $\gamma_2$ be $l$. By Definition~\ref{def-ooc-curve} it means that
\[
m\leq
 \frac
{\|f(x)\|}{\left(x^2 + \|f(x)\|^2\right)^{l/2}}\leq M \;,
\]
so that there exists a constant $C$ such that for all
sufficiently small $x$
 $$
m \leq \frac{\|f(x)\|}{|x|^l} \leq C \;,
$$
which means that $\gamma_1$ is sandwiched between two order $l$ paraboloids
${\mathcal P}_m$ and ${\mathcal P}_C$, respectively.

Let us now look at the order of contact of $\gamma_2$ with
$\gamma_1$. By the above, for any point $x \in X$
$$
d(x,{\mathcal P}_m) \le d(x,\gamma_1) \le d(x,{\mathcal P}_C) \;.
$$
Since ${\mathcal P}_m$ and ${\mathcal P}_C$ are order $l$ paraboloids, the
distances to them are uniformly equivalent to $|x|^l$, whence the claim
\end{proof}

\begin{remark}\rm
Our notion of the order of contact is more
general than the classical one. Obviously, if the order of contact of two
curves is $l$ in the classical sense, then it is also $l$ in our sense.
The converse, however, need not be true.
\end{remark}

\begin{remark}\rm
Without the $C^1$ assumption Proposition~\ref{symmetry} may fail. For instance, let us consider the graphs of the functions $\gamma_1(x)=0$ and $\gamma_2(x)=x^2+(1+\sin(1/x))x$ on the plane. $\gamma_2$ oscillates between $x^2$ and $x$ with the distance between consecutive crests of order $x^2$. Then, the order of contact of $\gamma_1$ with $\gamma_2$ in the sense of our Definition~\ref{def-ooc-curve} is 2, whereas the order of contact of $\gamma_2$ with $\gamma_1$ does not exists.
\end{remark}

Naturally, the $l$-th
order of contact (in the
sense of \cite{AVG-Z}, \cite{BW}, \cite{Hi}, \cite{HW}, \cite{U-V}) between
two $C^l$-curves is preserved under a $C^1$-diffeomorphism.
Lemma~\ref{preserv-ooc-lemma} below shows that the order of
contact in the sense of our Definition~\ref{def-ooc-curve} is preserved by
$C^1$-diffeomorphisms. In particular, our order of contact from
Definition~\ref{def-ooc-curve} does not depend on the choice of a coordinate
chart in the ambient manifold, so that it is well-defined in
the general manifold setup as well.

\begin{lemma}\label{preserv-ooc-lemma}
Given two $C^1$-curves in ${\mathbb R}^n$ intersecting at an isolated point, any
diffeomorphism of a neighborhood of this point preserves the order of contact
of these curves.
\end{lemma}

\begin{proof}{}
Let curves $\gamma_1$ and $\gamma_2$ have order of contact $l$ at some point $A$.
Then there are positive constants $m$ and $M$ such that for all points
$x\in\gamma_1$ sufficiently close to $A$
 \[
m\leq \frac {\Vert x - y\Vert }{\| x - A\|^{l}}
 \leq M,
 \]
where $y=y(x)$ denotes a point on $\gamma_2$ minimizing the distance from $x$
to $\gamma_2$. By the $C^1$ Mean  Value Theorem
 applied to the diffeomorphism $\phi$,
 \[
 \phi (x)- \phi (y) = \Bigg[  \int_{0}^{1}(D \phi )_{\sigma (s)} ds \Bigg] (x-y),
 \]
where $\sigma:[0,1]\to{\mathbb R}^n$ is a path connecting the points
$x=\sigma(0)$ and $y=\sigma(1)$. As $x \to A$, $y(x) \to A$. Then, $\sigma (s)
\to A$, and the matrix $ \int_{0}^{1}(D \phi )_{\sigma (s)} ds $ tends to the
invertible matrix $ (D \phi )_A $, hence the claim.
\end{proof}

Let us now pass to a discussion of the order of contact for higher dimensional
intersecting manifolds.

\begin{definition}\label{def-ooc-manifold} \rm
Let $S_1$ and $S_2$ denote two immersed $C^1$-manifolds in ${\mathbb{R}}^n$.
Suppose the two manifolds meet at an isolated point $A$. Their
{\it order of
contact} $l$ at $A$ is the supremum of the orders of contact of curves
$\gamma_1\subset S_1$ and $\gamma_2\subset S_2$ passing through point $A$.
\end{definition}

As it follows from Lemma~\ref{preserv-ooc-lemma}, the order of contact is
preserved under $C^1$-diffeomorphisms.

\bigskip

The well known ${\lambda}$-Lemma of Palis \cite{Pa} gives an important
description of chaotic dynamics. The basic assumption of this theorem is the
presence of a transverse homoclinic point, which means that the manifolds have
contact of order 1.

\begin{theorem}[$\lambda$-Lemma, Palis]\label{Palis}
Let $f$ be a diffeomorphism of ${\mathbb{R}}^n$ with a hyperbolic fixed
point at $0$ and $m$- and $p$-dimensional stable and unstable manifolds $W_S$
and $W_U$ ($m+p=n$). Let $D$ be a $p$-disk in $W_U$, and $w$ be another
$p$-disk in $W_U$ meeting $W_S$ at some point $A$ {\em transversely}. Then
$\bigcup_{n\geq 0} f^n(w )$ contains $p$-disks arbitrarily $C^1$-close to $D$.
\end{theorem}

Later in this section we will prove an analog of this lemma (Lemma~\ref{lemma})
for non-transversal homoclinic intersections.

For the estimates in the proofs of the Singular $\lambda$-Lemma (i.e., in the
case when the order of contact is greater than 1) and the Main Theorem we need
the following definitions. The first one, the definition of
a graph portion, deals with the shape of a global invariant manifold in the
vicinity of a homoclinic point. This homoclinic point must be chosen close to a
hyperbolic reference point. Then, the local properties of the invariant
manifold are checked. The second definition describes a group of
diffeomorphisms that can be
locally represented (after a $C^1$ change of coordinates) in
some special normal form.

\begin{definition}\label{def-p-p} \rm
Let $f$ be a diffeomorphism of ${\mathbb{R}}^n$ with a hyperbolic fixed point
at the origin and $U \in$ ${\mathbb{R}}^n$ be some small neighborhood of $0$. Denote by $W_S$ (resp., $W_U$) the
associated stable (resp., unstable) manifold, and by $m$ (resp., $p$) its
dimension ($m+p=n$). Let $A$ be a homoclinic point of $W_S$ and $W_U$. Denote
by $\mathcal V$ a small $p$-neighborhood in $W_U$ around the origin. Define a
local coordinate system $E_1$ at $0$, which spans $\mathcal V$. Similarly,
define a local coordinate system $E_2$ which spans $W_S$ near $0$. Let $E = E_1
+ E_2$, $E \subset U$. A $p$-neighborhood $\Lambda \subset W_U$ is a
{\it graph portion} in
$U$, associated with the homoclinic point $A$, if for some $n > 0$ $f^n(A) \in
\Lambda$, $\Lambda \subset (U \cap W_U)$, and $\Lambda$ is a graph in
$E$-coordinates of some $C^1$-function defined on $\mathcal V$.
\end{definition}

\begin{figure}[th]
\begin{center}
\begin{psfrags}
\psfrag{L}[][]{$\Lambda$}
\psfrag{A}[][]{$A$}
\psfrag{0}[][]{$0$}
\psfrag{V}[][]{$\mathcal V$}
\psfrag{WS}[][]{$W_S$}
\psfrag{WU}[][]{$W_U$}
     \epsfxsize=4.0in\leavevmode\epsfbox{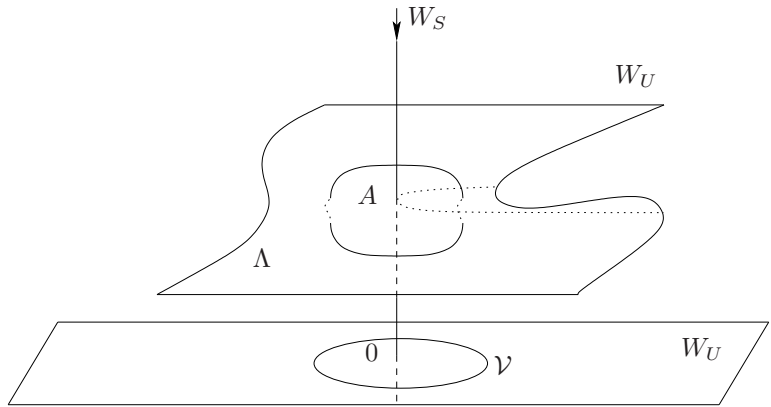}
\end{psfrags}
\end{center}
\caption{In this figure, $\Lambda$ is not a graph portion of the manifold $W_U$, because, for any $n$, an iteration of
$\Lambda$ with diffeomorphism $f^n$ is never a graph in the given coordinates.}
\label{fig_graph}
\end{figure}

\begin{definition} \label{defNormalForm} \rm
We will say that diffeomorphism $f$ with a hyperbolic fixed point $0$ satisfies
the {\it normal
form condition}, if there exists a local $C^1$ change of
coordinates in a neighbourhood of $0$,
under which $f$ takes the
following normal form: $f(x,y)=(S_1(x,y), S_2(x,y))$, where $S_1$ and $S_2$
belong to expanding and contracting subspaces respectively, and $S_2$
is linear, i.e $S_2(x,y) = {\mathcal B} y$, with $\Vert {\mathcal B} \Vert <1$.
\end{definition}

We are interested in the class of diffeomorphisms satisfying the normal form
condition described above. The next lemma
(Lemma~\ref{flattening}) allows us to ascertain belonging to this class by verifying certain
conditions (called resonances) on the eigenvalues of the linear part of the
map. The following two definitions are used in the lemma.

\begin{definition} \label{defPre5} \rm
Let $f$ be a diffeomorphism of ${\mathbb{R}}^n$ with
the linear part
$$( ( {\mathcal A} x )_1,\dots,( {\mathcal A}x )_p,( {\mathcal B} y )_1, \dots,( {\mathcal B} y )_m)
$$
at a hyperbolic fixed
point $0$, where $m$ and $p$ are the
dimensions of its stable and unstable manifolds, respectively.
Then,
$f$ satisfies the {\it mixed second order resonance condition} if there exist
$a \in \mathop{\rm spec}{\mathcal A}$ and $b \in
\mathop{\rm spec} {\mathcal B}$ with $a b \in (\mathop{\rm spec}
{\mathcal A}\cup \mathop{\rm spec} {\mathcal B}) $.
\end{definition}

\begin{definition} \label{def5} \rm
We will say that $f$ has
{\it mixed second order resonance in its contracting
coordinates} if there exist $a \in
\mathop{\rm spec}{\mathcal A}$, $b \in \mathop{\rm spec} {\mathcal
B}$ with $a b \in \mathop{\rm spec} {\mathcal B}$.
\end{definition}

\begin{lemma}\label{flattening}
If a $C^{\infty}$ diffeomorphism $f$ with a hyperbolic fixed point $0$ has no mixed second order resonances in the contracting
coordinates, then it satisfies the normal form condition.
\end{lemma}

\begin{remark}\rm
The $C^\infty$ assumption in the above Lemma~\ref{flattening} can be replaced with a $C^k$ assumption,
where $k$ depends on the spectrum of the linear part, like in Bronstein and Kopanskii~\cite{BK} (Theorem 11.9).
As this dependence is a bit complicated we
assume $C^\infty$.
\end{remark}

\begin{proof}{}
Let $x=(x_1,\dots,x_p)\in {\mathbb{R}}^p$, $y=(y_1,\dots,y_m)\in
{\mathbb{R}}^m$ ($p+m=n$) and $f(x,y): {\mathbb{R}}^n \to {\mathbb{R}}^n $ have the linear part
$$
(({\mathcal A}x )_1,\dots,({\mathcal A}x )_p,({\mathcal B}y )_1,
\dots,( {\mathcal B}y )_m),
$$
 with $\| {\mathcal A}^{-1}\|$,
$\| {\mathcal B} \| < {\lambda} <1$.

First, we will establish conjugation between $f$ and the following normal form:
$$
{\mathcal A}x + \left(
\sum_{{i=1,\dots,p;}\atop{j=1,\dots,m}}a_{ij}^1 x_i
y_j,\dots , \sum_{{i=1,\dots,p;}\atop{j=1,\dots,m}}a_{ij}^p
x_i y_j \right) ,
$$
$$
{\mathcal B}y + \left(
\sum_{{i=1,\dots,p;}\atop{j=1,\dots,m}}b_{ij}^1 x_i
y_j,\dots ,
\sum_{{i=1,\dots,p;}\atop{j=1,\dots,m}}b_{ij}^m
x_i y_j \right) .
$$
Conjugation of $f$ with this quadratic polynomial can be constructed
$C^1$-smooth, because any term of degree higher than 2 in the polynomial expansion, of a
$C^\infty$ diffeomorphism with a
hyperbolic fixed point satisfies the $S(1)$
condition defined by Bronstein and Kopanskii~\cite{BK}
(Definition 7.4, page 110).

Suppose that the polynomial expansion of $f$ contains a term $c
x^k y^l$ and at least one of the two inequalities holds:
$|k|>1$, $|l|>1$. For
definiteness let us assume that
$|k|>1$ (where $|k|$
denotes the length of the multiindex $k$). Also, we can assume that the
eigenvalues are ordered: ${\mathcal A}_1 \leq \dots \leq {\mathcal A}_p$. Let
$k=(k_1, \dots, k_p)$ and $j = \max\{i: k_i \neq 0\}$. Then,
$$
{\mathcal A}_1^{k_1}\cdots {\mathcal A}_j^{k_j} >
{\mathcal A}_j,
$$
i.e. $c x^k y^l$ satisfies the $S(1)$ condition.

Remark 7.6 in Bronstein and Kopanskii~\cite{BK} (page 111) asserts that $S(1)
\implies {\mathcal A}(1)$. Then, Samovol's theorem (\cite{BK}, Theorem 10.1,
page 179) implies $C^1$-conjugation of $f$ with
the quadratic polynomial.

Moreover, since $f$ has no mixed second order resonances in
its contracting coordinates, by Samovol's theorem it can
actually be $C^1$-conjugated to
\begin{align*}
{\mathcal A}x& + \left(
\sum_{{i=1,\dots,p;}\atop{j=1,\dots,m}}a_{ij}^1 x_i
y_j,\dots , \sum_{{i=1,\dots,p;}\atop{j=1,\dots,m}}a_{ij}^p
x_i y_j \right),
\\
{\mathcal B}y.
\end{align*}
\end{proof}

\bigskip

Obviously, the conclusion of the $\lambda$-Lemma of Palis is not true for an
arbitrary degenerate (non-transverse) crossing (see, for example,
figure~\ref{newhousePic}).

We will now prove an analog of the ${\lambda}$-Lemma for the non-transverse
case in ${\mathbb{R}}^n$.

\begin{lemma}[Singular $\lambda$-Lemma]\label{lemma}
Let $f$ be a $C^2$ diffeomorphism of ${\mathbb{R}}^n$ with a hyperbolic fixed
 point at $0$ and with $m$- and $p$-dimensional stable and unstable manifolds
$W_S$ and $W_U$ ($m+p=n$), respectively. Assume that $f$
satisfies the normal form condition and there exists $\Lambda
\subset W_U$, a graph portion in a
small neighborhood of $0$, associated with some degenerate homoclinic point $A$
(cf. Definition~\ref{def-p-p}). Also assume that $l$, the
order of contact at $A$, is finite ($1 < l <\infty$).

Then, for any $\rho >0$, for an arbitrarily
 small $\epsilon$-neighborhood ${\mathcal U} \subset
{\mathbb R}^n$ of the origin and for the graph portion
${\Lambda}$, the set $(\bigcup_{n\geq 0} f^n({\Lambda}
))\setminus {\mathcal U}$ contains disks $\rho$-$C^1$ close to ${\mathcal V}
\setminus {\mathcal U}$ (here ${\mathcal V}$ is a small neighborhood of $0$ in
the local unstable manifold).
\end{lemma}

\begin{figure}[th]
\begin{center}
\begin{psfrags}
\psfrag{L}[][]{$\Lambda$}
\psfrag{f}[][]{$f(\Lambda)$}
\psfrag{ff}[][]{$f^2(\Lambda)$}
\psfrag{A}[][]{$A$}
\psfrag{f(A)}[][]{$f(A)$}
\psfrag{f(f(A))}[][]{$f^2 (A)$}
\psfrag{V}[][]{$\mathcal V$}
\psfrag{S}[][]{$W_S$}
\psfrag{U}[][]{$W_U$}
\psfrag{X}[][]{$X$}
\psfrag{Y}[][]{$Y$}
\psfrag{o}[][]{$0$}
     \epsfxsize=3.0in\leavevmode\epsfbox{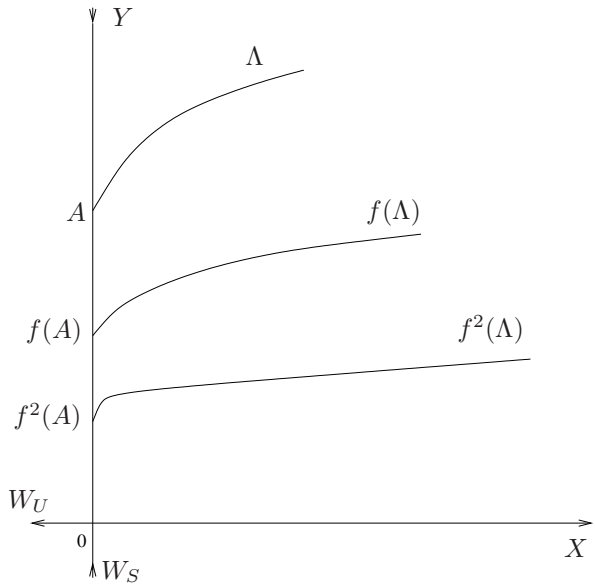}
\end{psfrags}
\end{center}
\caption{Iterations of the graph portion $\Lambda$ with the diffeomorphism $f$.}
\label{lemmapict}
\end{figure}

\begin{proof}{}
Let $\alpha = 1/l$ ($0 < \alpha <1$). Since ${\Lambda}$ is a graph portion that
has order of contact $l$ with $W_S$ at point $A$, we can assume that locally
${\Lambda}$ is the
graph of a function
\[
{\Lambda}(x) = A+r(x):{\mathbb{R}}^p \to {\mathbb{R}}^n, \quad r(0)=0,
\]
and for any sufficiently small $\sigma >0$
\[
 | r(x)| \leq \mathop{\rm const} \cdot | x|^\alpha\quad \mbox{and}\quad
 \left|  \di r(x) \right|
 \leq \mathop{\rm const} \cdot | x|^{\alpha -1}\
\]
for all $| x| < \sigma$, $ i=1,\dots,p$. Let $x=(x_1,\dots,x_p)\in {\mathbb{R}}^p$,
$y=(y_1,\dots,y_m)\in {\mathbb{R}}^m$ ($p+m=n$) and $f(x,y): {\mathbb{R}}^n \to
{\mathbb{R}}^n $ have the linear part
$$
(({\mathcal A}x )_1,\dots,({\mathcal A}x )_p,({\mathcal B}y )_1,
\dots,( {\mathcal B}y )_m).
$$
Assume that $\| {\mathcal A}^{-1}\|$, $\| {\mathcal B} \| < {\lambda} <1$.
Choose an arbitrarily small $\Delta$. We know that locally $f$ can be written
in the form $f(x,y)=(S_1(x,y), S_2(x,y))$, where $S_2(x,y)$ is linear,
$S_2(x,y) = {\mathcal B} y$. Since $f$ is $C^2$, the existence of invariant
manifolds implies that, in appropriate coordinates and in a sufficiently small
neighborhood of the origin, $f$ can be written as
\begin{align*}
S_1(x,y)= {\mathcal A}x& +
\left(
\sum_{i=1,\dots,p}x_i
U_i^1 (x,y),\dots ,
\sum_{i=1,\dots,p}x_i U_i^p (x,y)
\right),\\
S_2(x,y)= {\mathcal B} y,
\end{align*}
with $U(0) = 0$, $\| U\|_{C^0}\leq \Delta$ (for an arbitrary small $\Delta$), and
$\| U\|_{C^1}$ bounded.

Consider $f(x, {\Lambda}(x))=(T_1^{\Lambda}(x),T_2^{\Lambda}(x))$. We will work
with $(x,T_2^{\Lambda} \circ (T_1^{\Lambda})^{-1}(x))$ and deduce that $f^n(x,
{\Lambda}(x))$ is $C^1$-small for $n$ big enough, $|x|\in (\epsilon , \sigma
)$, and $\sigma >0$ sufficiently small. First we will show that, in
$C^1$-topology, $(T_1^{\Lambda})^{-1}$ is $\Delta$-close to ${\mathcal
A}^{-1}$. For simplicity, we will denote $T_1^{\Lambda}$ by $T_1$ and
$T_2^{\Lambda}$ by $T_2$. Then
\[
T_1(x)=
{\mathcal A}x + \left( \sum_{i=1,\dots,p} x_i U_i^1 (x,{\Lambda}(x)),\dots, \sum_{i=1,\dots,p} x_i U_i^p (x,{\Lambda}(x))
\right).
\]

\begin{claim}\label{claim_in}
\[
\bigl\| T_1^{-n}\bigr\|_{C^1} \leq \left( \bigl\| A^{-1}\bigr\|_{C^1} + K\cdot\Delta\right)^n
\] for $| x | < \sigma$  ($\sigma >0 $ sufficiently small, $K >0$).
\end{claim}

\begin{proof}{}
Fix some $l\in \{1,\dots,p\}$. Recall that $\Lambda (x) = A + r(x)$.
\begin{align*}
\bigg| x_i \dl {\Lambda}_j(x) \bigg|
&\leq | x_i| \cdot
\bigg| \dl {\Lambda}_j(x) \bigg| \\
&\leq | x| \cdot O(1) | x|^{\alpha-1}\\
&\leq O(1)| x|^\alpha
\end{align*}
Through the proof of this Theorem, $O(1)$ will denote the set of functions
\begin{align*}
O(1) =& \big\{ \gamma : {\mathbb{R}} \mapsto
{\mathbb{R}}
 \mbox{, such that there exists a positive constant $c$ with }\\
&| \gamma (\zeta )| \leq c \mbox{ for all sufficiently small }\zeta \big\}
\end{align*}
Also,
\begin{align*}
|x|\cdot
\bigg| \dl
U_i^t(x,{\Lambda}(x)) \bigg| &=
|x|\cdot \bigg| \dl U_i^t(x,y) +
\sum_{k=1}^{m}
\dk U_i^t(x,y) \cdot \dl {\Lambda}_k(x) \bigg| \\
&= O(1)| x|^\alpha,
\end{align*}
and
$$
\big\| U (x,{\Lambda}(x))\big\|_{C^0} \leq \Delta.
$$

Therefore,
\begin{align*}
& \Bigg\|
\Bigg(
\sum_{i=1,\dots,p}x_i U_i^1
(x,{\Lambda}(x)),\dots ,\sum_{i=1,\dots,p}x_i U_i^p
(x,{\Lambda}(x))
\Bigg)
\Bigg\|_{C^1}\leq \Delta \cdot O(1),
\end{align*}
if $\sigma$ is sufficiently small and $| x| < \sigma$.
(An arbitrarily small $\Delta$ was chosen earlier). Then,
\[
\bigl\| T_1^{-n}\bigr\|_{C^1} \leq \left( \bigl\| A^{-1}\bigr\|_{C^1} + K\cdot\Delta\right)^n
\]
\end{proof}

Now, we continue the proof of Lemma~\ref{lemma} and show that $\|T_2^n \circ
T_1^{-n} \|_{C^1} $ is $\rho$-small.

\begin{align*}
\|T_2^n \circ T_1^{-n} \|_{C^1}
& = \|{\mathcal B}^n \cdot \Lambda (T_1^{-n})\|_{C^1}\\
&\leq \|{\mathcal B}\|^n \cdot \| \Lambda (T_1^{-n})\|_{C^1}\\
& = O(1)\cdot \| (T_1^{-n})^{\alpha -1}\|_{C^0}\cdot \| T_1^{-n}\|_{C^1}\\
& = O(1)\cdot \| T_1^{-n}\|_{C^0}^{\alpha -1}\cdot \| T_1^{-n}\|_{C^1}\\
& = O(1)\cdot \big(\| T_1^{-n}\|_{C^1} \| x \|\big)^{\alpha -1}\cdot \| T_1^{-n}\|_{C^1}\\
& = O(1)\cdot \| T_1^{-n}\|_{C^1}^{\alpha}\cdot \| x \|^{\alpha -1}.
\end{align*}
Applying Claim~\ref{claim_in} to the last expression, we obtain
\[
\|T_2^n \circ T_1^{-n} \|_{C^1} = O(1)\cdot \Big(\big\| A^{-1} \big\| + \Delta K \Big)^{\alpha n}\big\| x \big\|^{\alpha -1}.
\]
The last expression is small if $\big\| x \big\| > \epsilon$ and $n$ is big enough.
\end{proof}

\begin{figure}[th]
\begin{center}
\begin{psfrags}
\psfrag{q}[][]{$A$}
\psfrag{f(f(q))}[][]{$f^2 (A)$}
\psfrag{f(q)}[][]{$f(A)$}
\psfrag{p}[][]{$0$}
     \epsfxsize=2.55in\leavevmode\epsfbox{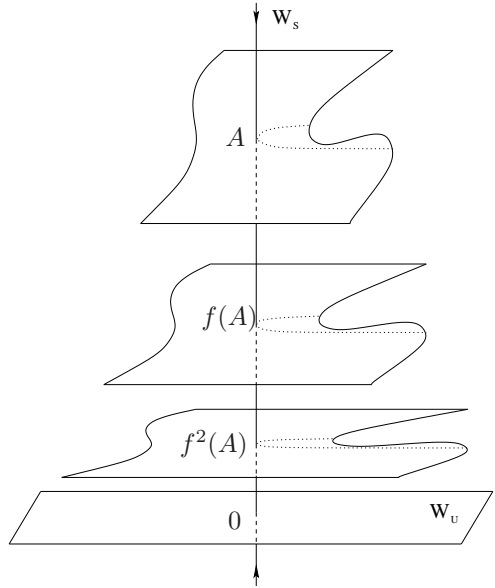}
\end{psfrags}
\end{center}
\caption{The iterated manifold is not a graph portion. The iterations of this manifold do not come $C^1$-close to the neighborhood of 0 in $W_U$.}
\label{lemma_counter_ex}
\end{figure}

\begin{remark}\rm
Clearly, if $W_U$ does not contain any graph portion $\Lambda$, the conclusion
of the $\lambda$-Lemma~\ref{lemma} is not true. Figure~\ref{lemma_counter_ex}
illustrates this situation.
\end{remark}

\section{The Main Theorem}\label{4}

In this section, we extend the results of Hirsch \cite{Hi},
Homburg and Weiss \cite{HW}, and  Gavrilov and Shilnikov
\cite{GS1}, \cite{GS2} to certain types of non-planar homoclinic tangencies with
finite orders of contact. Our goal is to establish a transversal crossing
arbitrarily close to the non-transversal one. Then, it would follow that there
exist horseshoes, located arbitrarily close to the degenerate crossing.

Consider a diffeomorphism $f$ on ${\mathbb{R}}^n$ which has $1$-dimensional
stable and \mbox{$(n-1)$}-dimensional unstable manifolds, and a homoclinic point $B$. Since the stable manifold is $1$-dimensional, we can obviously define one- and two-sided homoclinic intersections.

\begin{definition} \label{newhouseDef}\rm
We will say that $f$ has a homoclinic tangency of {\it Newhouse type}, if $f$
possesses only one-sided homoclinic tangencies.
\end{definition}
If $f$ possesses only one-sided homoclinic tangencies, then transversal
intersection need not exist in the neighborhood of $B$ even in the planar case.
The Newhouse example (Figure~\ref{newhousePic}) illustrates this situation. In
the planar case, this type of tangency (Newhouse type) was studied by Homburg
and Weiss \cite{HW}. Unfortunately, we can not extend their results to
${\mathbb{R}}^n$.

\begin{figure}[th]
\begin{center}
\begin{psfrags}
\psfrag{p}[][]{$p$}
\psfrag{q}[][]{$B$}
     \epsfxsize=2.8in\leavevmode\epsfbox{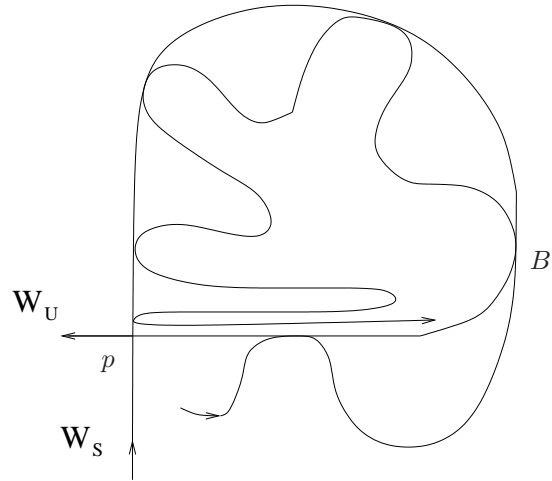}
\end{psfrags}
\end{center}
\caption{Newhouse example.}
\label{newhousePic}
\end{figure}

The other three cases of one-sided intersections in ${\mathbb{R}}^2$ were
studied by Gavrilov and Shilnikov, and Homburg and Weiss (Figure~\ref{types_of_crossing}, cases I, II, III). The case of two-sided intersection in ${\mathbb{R}}^2$ ((Figure~\ref{types_of_crossing}, case IV) was investigated by Hirsch and Chirchill and Rod.

We will consider, in higher dimensions, the situation similar to the planar cases I, II, III and IV, although we clearly have a different number of cases in ${\mathbb{R}}^n$. For example, cases I and II (with 1-dimensional stable, and 2-dimensional unstable manifolds) are the same in ${\mathbb{R}}^3$. But this is not important for our proof.

If a one-sided intersection is preceded or followed by a two-sided intersection
(see Figure~\ref{types_of_crossing}, cases I, II, III), then we can assume that
a two-sided intersection is arbitrarily close to a one-sided touching.
Consequently, it is enough to find a transversal crossing arbitrarily close to
a two-sided crossing.

\begin{figure}[th]
\begin{center}
\begin{psfrags}
\psfrag{1}[][]{case I}
\psfrag{2}[][]{case II}
\psfrag{3}[][]{case III}
\psfrag{4}[][]{case IV}
\psfrag{B}[][]{$B$}
     \epsfxsize=4.3in\leavevmode\epsfbox{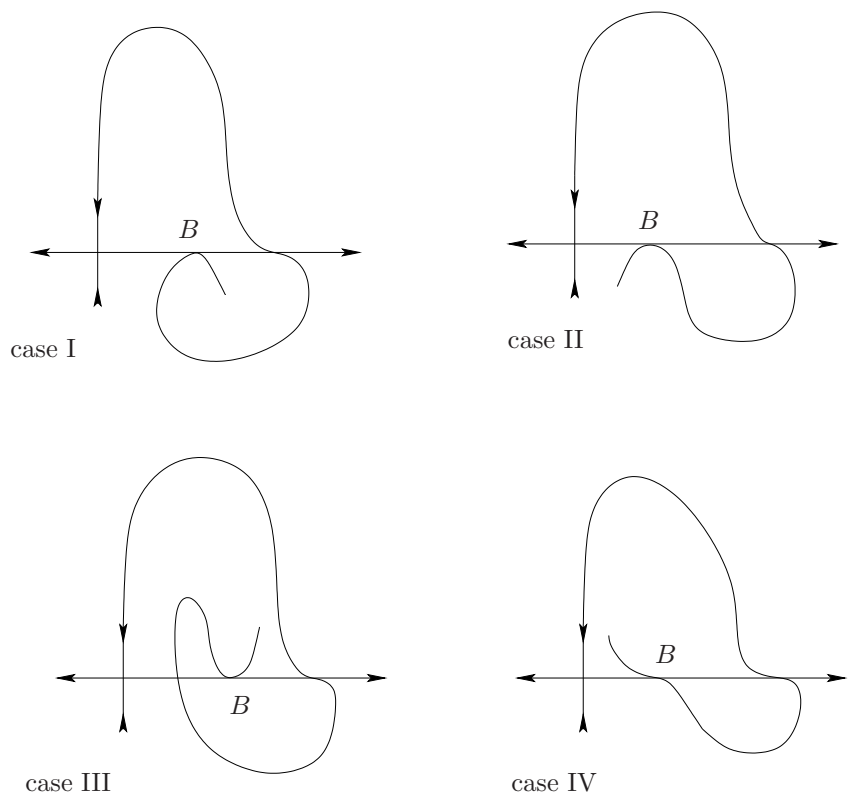}
\end{psfrags}
\end{center}
\caption{Two-sided intersections of invariant manifolds (non-Newhouse types).}
\label{types_of_crossing}
\end{figure}

\begin{theorem}\label{main-thm}
Let $M$ be a smooth manifold, and let $f:M \to M$ be a diffeomorphism with a
hyperbolic fixed point $p \in M$ and with stable and unstable invariant
manifolds $W_S$ and $W_U$ at $p$. Assume:
\begin{itemize}
\item[(i)] $f$ is locally $C^1$-linearizable in a neighborhood of $p$;
\item[(ii)] $W_U$ and $W_S$ have complementary dimensions $n-1$ and $1$;

\item[(iii)] $W_U$ and $W_S$ have a finite tangential contact at an isolated homoclinic point $B \in M$ ($B\neq p$),
which is not of Newhouse type;
\item[(iv)] $W_U$ contains a graph portion $\Lambda$ associated with $B$ in a small neighborhood, contained in the
neighborhood of linearization.
\end{itemize}
Then the invariant manifolds have a point of transverse intersection
arbitrarily close to the point $B$. Therefore, $f$ has a
horseshoe (arbitrarily close to $B$) according to the Birkhoff-Smale Theorem.
\end{theorem}

\begin{proof}{}
Consider a $C^1$ linearization $\phi: N'\mapsto N$ from a neighborhood
$N'\subset M$ of $p$ onto a neighborhood $N\subset {\mathbb{R}}^n$ of the
origin. We choose a linearization such that $f$ has a
normal Jordan form in the new coordinates. We may assume that $\phi$ takes the
local stable manifold at $p$ to a neighborhood in the $Y$-axis, and the local
unstable manifold to a neighborhood in the $X$-hyperplane.

There are images of $\phi (B)$ in $N$, under forward and backward iterates of
$f$, on both the local stable and local unstable manifolds of $0$.
There is a compact curve,
$\Gamma \subset N$ (see Figure~\ref{fig_proof-illustration}), with the
 following properties:
\begin{itemize}
\item $\Gamma $ is an image under $f$ of the stable manifold;
\item $\Gamma$ has finite contact with the $X$-plane at some isolated point \\$A = (A_1,...,A_{n-1},0)$;
\item $\Gamma$ has two-sided intersection with the $X$-plane at $A$;
\item

$A$ is arbitrarily close to the origin.
\end{itemize}

Any $C^1$-curve $\Gamma$ immersed in ${\mathbb{R}}^n$, which has finite order of contact with the $X$-hyperplane at point $A=(A_1,...,A_{n-1},0)$, locally can be parameterized as:
\[
\Gamma =
\left\{\begin{array}{lll}
u^1 & = & a_1(\tau ) +A_1\\
.&.&.\\
.&.&.\\
.&.&.\\
u^{n-1} & = & a_{n-1}(\tau )+A_{n-1}\\
v & = & {\tau}^l,
\end{array}
\right.
\]
\\
for some constants $l > 1$, $A_j$, and $C^1$ functions $a_j$, such that $a_j
(0)=0$ and $a_j '(\tau ) \leq r$ in a small neighborhood of the origin, $\max_j
|A_j| > 0$ ($j=1,...,n-1$).

Since $\Lambda$ is a graph portion, it can be parameterized as:
\[
 \Lambda =
\left\{\begin{array}{lll}
x^1 & = & t_1 \\
.   & . & .\\
.   & . & .\\
.   & . & .\\
x^{n-1} & = & t_{n-1}\\
y & = & \phi (t_1,...,t_{n-1}) +B,
\end{array}
\right.
 \]
for some constant $B \neq 0$ and $C^1$ function $\phi$, $\phi (0) =0$ .
\\

Now it suffices to refer
to the following lemma, in order to have the proof of this theorem completed.

\begin{lemma}\label{formula}

Suppose
\[
 f:(x,y) \to (\beta_1 x^1 + \sigma_2 x^2, \beta_2 x^2 + \sigma_3 x^3,...,\beta_{n-2} x^{n-2}+ \sigma_{n-1} x^{n-1},\beta_{n-1} x^{n-1}, \alpha y),
 \]

 \[
  0<\alpha <1 <\beta_j \quad (j=1,...,n-1)
\]
and $\sigma_k$ are $0$ or $1$ -- according to the Jordan normal
form of $f$.

Then there exists a natural number $k_*$ such that
$ f^k\Lambda$ crosses  $\Gamma$  transversely  if $ k>k_*.$

\end{lemma}

\begin{figure}[th]
\begin{center}
\begin{psfrags}
\psfrag{xyn}[][]{$(x_n, y_n)$}
\psfrag{xy1}[][]{$(x_1, y_1)$}
\psfrag{f(xyn)}[][]{$f^n(x_n, y_n)$}
\psfrag{f(xy1)}[][]{$f^n(x_1, y_1)$}
\psfrag{T}[][]{$f$}
\psfrag{f(l)}[][]{$f(\Lambda)$}
\psfrag{f(f(l))}[][]{$f^2(\Lambda)$}
\psfrag{fn(l)}[][]{$f^n(\Lambda)$}
\psfrag{G}[][]{$\Gamma$}
\psfrag{o}[][]{$0$}
\psfrag{A}[][]{$B$}
\psfrag{aA}[][]{$\alpha B$}
\psfrag{a2A}[][]{$\alpha^2 B$}
\psfrag{anA}[][]{$\alpha^n B$}
\psfrag{B}[][]{$A$}
\psfrag{l}[][]{$\Lambda$}
\psfrag{X}[][]{$X$}
\psfrag{Y}[][]{$Y$}
     \epsfxsize=4.0in\leavevmode\epsfbox{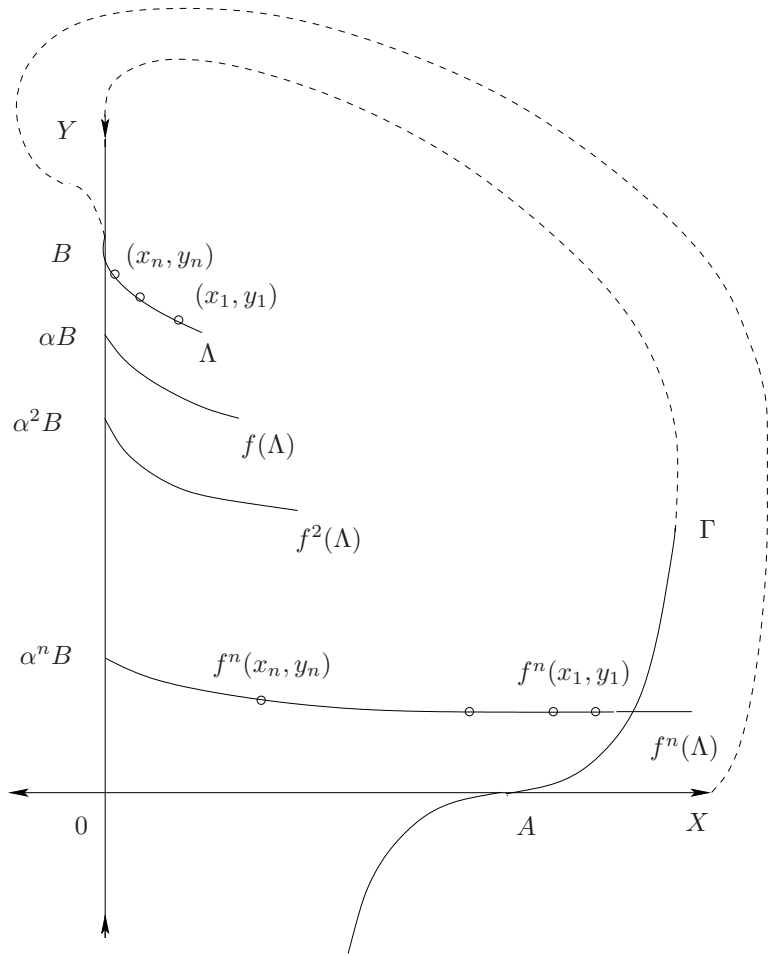}
\end{psfrags}
\end{center}\vspace{-0.1in}
\caption{Degenerate homoclinic crossing near the fixed point $0$. Illustration of the proof.}
\label{fig_proof-illustration}
\end{figure}

\begin{proof}{}
To simplify notations, first assume that $f$ has a diagonal Jordan normal form,
i.e., all $\sigma_k$ are 0. Consider the iterations of $\Lambda$ by $f$:
\[
f^{k_i}\Lambda =
\left\{\begin{array}{lll}
u^1 & = & \beta _1^{k_i}t_1\\
.   & . & .\\
.   & . & .\\
.   & . & .\\
u^{n-1} & = & \beta_{n-1}^{k_i}t_{n-1}\\
v & = & \alpha^{k_i} \left( \phi  (t_1,...,t_{n-1}) + B\right)\\
\end{array}
\right.
\]

We can find sequences  $(x_i,y_i) \in \Lambda$ and $k_i \rightarrow \infty$
such that $(x_i,y_i)\rightarrow  (0,B)$ and $f^{k_i}(x_i,y_i)\rightarrow
(A,0)$. This implies that $f^{k_i}\Lambda $ and $\Gamma $ will cross at some
point other than $(0,B)$ and $(A,0)$ (see Figure~\ref{fig_proof-illustration}).
Thus, we only need to show that this intersection is transversal.

Consider the matrix, constructed in the following way:
\begin{itemize}
\item The first $n-1$ rows represent the linear space of the tangent hyperplane
to the $(n-1)$-manifold $f^{k_i} \Lambda $.
\item The last row represents the linear space of the
tangent line to the curve $\Gamma$.
\end{itemize}
In order to prove that $f^{k_i} \Lambda $ crosses $\Gamma$ transversely at some
points of the form
$(u^1(t_i),...,u^{n-1}(t_i),v(t_i))$, we show that the determinant of the
matrix below is never $0$ for ${k_i}$ sufficiently large. We consider the
following determinant $D_i$, which corresponds to the $k_i$-th iteration of
$f$:
\[
D_i=
\left|\begin{array}{ccccc}
\beta_1^{k_i}&\nonumber&0&\nonumber&\alpha^{k_i}\frac{\partial}{\partial t_1}\phi(t_i)
\\
\nonumber&\nonumber&\nonumber&\nonumber&\nonumber
\\

\nonumber&\nonumber&\nonumber&\nonumber&\nonumber
\\
\nonumber&\ddots&\nonumber&\nonumber&\vdots
\\
\nonumber&\nonumber&\nonumber&\nonumber&\nonumber
\\
0&\nonumber&\beta_{n-1}^{k_i}&\nonumber&\alpha^{k_i}\frac{\partial}{\partial t_{n-1}}\phi(t_i)
\\
\nonumber&\nonumber&\nonumber&\nonumber&\nonumber
\\
\nonumber&\nonumber&\nonumber&\nonumber&\nonumber
\\
a_1 '(\tau _i)&\ldots &a_{n-1} '(\tau_i)&\nonumber&l \tau_i^{l-1}
\\
\end{array}\right|
 \]

Since $\Lambda$ intersects $\Gamma$,
\[
a_j(\tau _i) + A_j = \beta_j^{k_i}t_i, \ \ \ \ \ \ \ \ j= 1, ..., n-1
\]
and
\[
\tau _i^l = \alpha ^{k_i}(\phi(t_i) + B).
\]
Then, it follows that
\[
\tau ^{l-1} = \alpha ^{k_i\frac{l-1}{l}}\left( \phi(t(\tau)) + B \right)^{\frac{l-1}{l}}
\]
for any small value of the parameter $\tau$, and the corresponding small value
of the parameter $t(\tau )$.

We substitute the above expression in the matrix $D_i$ and calculate its determinant:
\[
D_i= \beta_1^{k_i}...\beta_{n-1}^{k_i}\Bigg[
 (-1)^{n-1}l \alpha ^{k_i\frac{l-1}{l}}(\phi (t(\tau))
+ B)^{\frac{l-1}{l}}
\]
\[\left.
 + \sum_{j=1}^{n-1}\alpha ^{k_i} \frac{\partial}{\partial
t_j}\phi(t(\tau))(-1)^j a_j'(\tau )\beta _j^{-k_i} \right]
\]

We can assume that $\phi (t(\tau)) + B > B/2$. Also, $a_j'(\tau)$ in our
parameterization of $\Gamma$ are bounded from above by $r$.

By Lemma~\ref{lemma} for any $\epsilon >0$ there exists
$k_*$ such that for any $k > k_*$
\[
\alpha ^k \frac{\partial}{\partial t_j}\phi(t(\tau)) < \epsilon,
\]
for all $j = 1, ... n-1$.
Then
\[
| D_i | > \beta_1^{k_i}...\beta_{n-1}^{k_i}
\left[
l \alpha ^{k_i\frac{l-1}{l}}(B/2)^{\frac{l-1}{l}}
 -
\alpha ^{k_i} \epsilon (n-1) r \beta _j^{-k_i}
\right]
\]
\[
= \beta_1^{k_i}...\beta_{n-1}^{k_i}\alpha ^{k_i\frac{l-1}{l}}
\left[
l (B/2)^{\frac{l-1}{l}}
 -
\alpha ^{k_i/l} \epsilon (n-1) r \beta _j^{-k_i}
\right]
\]
Take
\[
\epsilon < \frac{l(B/2)^{\frac{l-1}{l}}}{r(n-1)}
\]
Then,
\[
\left[
l (B/2)^{\frac{l-1}{l}}
 -
\alpha ^{k_i/l} \epsilon (n-1) r \beta _j^{-k_i}
\right]
\]
is positive for any sufficiently big iteration $k_i$.

\bigskip
If $f$ is non-diagonalizable, using the Jordan normal form, one can see that
\[
D_i= \beta_1^{k_i}...\beta_{n-1}^{k_i}
\left[
(-1)^{n-1}l \alpha ^{k_i\frac{l-1}{l}}(\phi (t(\tau)) + B)^{\frac{l-1}{l}}
\right.
\]
\[
\left.
 +
\sum_{j=1}^{n-1}\alpha ^{k_i} \frac{\partial}{\partial t_j}\phi(t(\tau))(-1)^j a_j'(\tau )\beta _j^{-k_i}
\right.\]
\[\left.
- \sum_{j=1}^{n-1}\alpha ^{k_i} \frac{\partial}{\partial
t_j}\phi(t(\tau))P_{k_i -1}(\beta _j) \beta _j^{-k_i} \right],
\]
where $P_{k_i -1}(\beta _j)$ are polynomial functions of degree
$k_i -1$.
\\
Then,
\[
| D_i | > \beta_1^{k_i}...\beta_{n-1}^{k_i}\alpha ^{k_i\frac{l-1}{l}}
\left[
l (B/2)^{\frac{l-1}{l}}
 -
\right.\]
\[\left.
\left(\alpha ^{k_i/l} \epsilon (n-1) r \beta _j^{-k_i}
+
\alpha ^{k_i/l} \epsilon (n-1) \left|P_{k_i -1}(\beta _j)\right| \beta _j^{-k_i}
\right)\right]
\]
and
\[
\left[
l (B/2)^{\frac{l-1}{l}}
 -
\left(\alpha ^{k_i/l} \epsilon (n-1) r \beta _j^{-k_i}
+
\alpha ^{k_i/l} \epsilon (n-1) \left|P_{k_i -1}(\beta _j)\right| \beta _j^{-k_i}
\right)\right]
\]
is also non-zero, because $\alpha ^{k_i/l}$ tends to $0$ exponentially, as $k_i\to\infty$.

This implies that for some $k_*$, the iterations $f^{k_i}$
($k_i> k_*$) applied to the manifold $\Lambda$ near the origin will produce a
transverse crossing with $\Gamma$.

\end{proof}

Applying Lemma~\ref{formula}, we complete the proof of
Theorem~\ref{main-thm}.
\end{proof}

\begin{figure}[th]
\begin{center}
\begin{psfrags}
\psfrag{p}[][]{$0$} \psfrag{B}[][]{$A$} \psfrag{U}[][]{$W_U$}
\psfrag{S}[][]{$W_S$} \psfrag{L}[][]{$\Lambda$}
\psfrag{fL}[][]{$f(\Lambda)$} \psfrag{ffL}[][]{$f^2(\Lambda)$}
\psfrag{G}[][]{$\Gamma$}
     \epsfxsize=2.55in\leavevmode\epsfbox{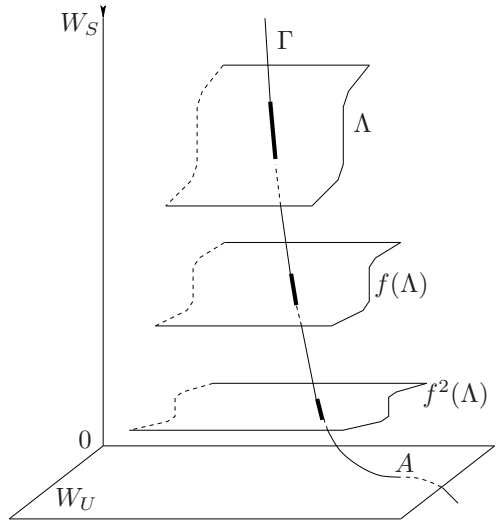}
\end{psfrags}
\end{center}
\caption{$\Lambda$ is not a graph portion. Intersections of
$\Gamma$ with $f^n(\Lambda)$ are indicated by a thicker line.
$\Gamma$ and $f^n(\Lambda)$ have no transversal intersections.}
\label{theorem_counter_ex}
\end{figure}

\begin{remark}\rm
It is clear from the proof of Theorem~\ref{main-thm} that, if $W_U$ does not
contain any graph portion, then $W_U$ possibly has no transversal crossings
with $\Gamma$. Figure~\ref{theorem_counter_ex} shows that $\Gamma \cap
f^n(\Lambda)$  can be a 1-dimensional manifold.
\end{remark}

 \medskip

 Received October 2003; revised September 2004.
 \medskip


\begin{thebibliography}{00}

\bibitem{AVG-Z}
V. I. Arnold, S. M. Guse\u\i n-Zade, A. N. Varchenko,
Singularities of differentiable maps. Vol. I. The classification of critical
points, caustics and wave fronts, Birkh\"{a}user, Boston, 1985.

\bibitem{BK}
I. U. Bronstein, A. Ya. Kopanskii, Smooth Invariant
Manifolds and Normal Forms, World
Scientific, Singapore, 1994.

\bibitem{BW}
K. Burns, H. Weiss, {\em A geometric criterion for positive topological
entropy}, Comm. Math. Phys., 172 (1995), 95-118.

\bibitem{CR}
R. Churchill and D. Rod, {\em Pathology in dynamical systems. III. Analytic
Hamiltonians}, J. Diff.
Equations, 37 (1980), 23-38.

\bibitem{GS1}
N. Gavrilov, L. Silnikov {\em On 3-dimensional dynamical systems close to
systems with a structurally unstable homoclinic curve. I}, Math. USSR Sb., 88 (1972), 467-485.

\bibitem{GS2}
N. Gavrilov, L. Silnikov {\em On 3-dimensional dynamical systems close to
systems with a structurally unstable homoclinic curve}, Math. USSR Sb., 90 (1973), 139-156.

\bibitem{GS}
S. Gonchenko, L. Shilnikov {\em Dynamical systems with structurally unstable
homoclinic curves}, Soviet Math. Dokl., 33 (1986), 234-238.

\bibitem{Ha}
P. Hartman, {\em On local homeomorphisms of Euclidean spaces},
Bol. Soc. Mat. Mexicana (2), 5 (1960), 220-241.

\bibitem{Hi}
M. Hirsch, {\em Degenerate homoclinic crossings in surface diffeomorphisms}, preprint (1993).

\bibitem{HW}
A. J. Homburg, H. Weiss, {\em  A geometric criterion for positive topological entropy. II: Homoclinic tangencies},
Comm. Math. Phys., 208 (1999), 267-273.

\bibitem{K1}
A. Katok, {\em Lyapunov exponents, entropy and periodic orbits for
diffeomorphisms}, Publ. Math. IHES, 51 (1980), 137-173.


\bibitem{K2}
A. Katok, {\em Nonuniform hyperbolicity and structure of smooth dynamical
systems}, Proc. International Congress of Mathematicians, Warszawa, 2 (1983), 1245-1254.

\bibitem{M}
J. Moser, {\em The analytic invariants of an area-preserving mapping near a
hyperbolic fixed point}, Comm. Pure
Appl. Math., 9
(1956), 673-692.

\bibitem{NP}
S. Newhouse and J. Palis, {\em Cycles and bifurcation theory}, Asterisque, 31, Soc. Math. France, Paris, 43-140 (1976).

\bibitem{Pa}
J. Palis, {\em On Morse-Smale dynamical systems},
Topology, 8 (1969), 385-405.

\bibitem{R}
V. Rayskin, {\em Multidimensional singular $\lambda$-Lemma}, Electron. J. Diff. Eqns., 38 (2003), 1-9.

\bibitem{S}
S. Smale, {\em Diffeomorphisms with many periodic points},
in:Differential and Combinatorial topology (edited by S. S.
Cairnes), Princeton University Press, 63-80, (1965).

\bibitem{U-V}
R. Uribe-Vargas, {\em Four-vertex theorems in higher dimensional spaces for a
larger class of curves than the convex ones}, C.R. Acad. Sci. Paris,
Ser. I, 330 (2000), 1085-1090.

\end{thebibliography}
\end{document}